\documentclass[11pt]{article}

\usepackage{amssymb,amsmath,amsthm}
\usepackage{epsfig}
\usepackage{breqn}
\usepackage{rotating}
\usepackage{amsfonts}
\usepackage{setspace}
\usepackage{fullpage}
\usepackage{enumitem}
\usepackage{bbold} 
\usepackage{comment}
\usepackage{pgf,tikz}
\usepackage{mathtools}  
\usepackage{hyperref}
\mathtoolsset{showonlyrefs} 
\usepackage{authblk}

\usepackage{setspace}

\newtheorem{theorem}{Theorem}
\newtheorem{claim}{Claim}
\newtheorem{lemma}{Lemma}
\newtheorem{corollary}{Corollary}
\newtheorem{definition}{Definition}

\newtheorem{remark}{Remark}

\newcommand{\floor}[1]{\left\lfloor{#1}\right\rfloor}

\DeclareMathOperator{\ex}{ex}

\def\h{\mathcal H}
\def\P{\mathcal P}

\def\S{\mathcal S}

\def\K{\mathcal K}
\def\BC{\mathcal {BC}}
\def\BP{\mathcal {BP}}
\def\F{\mathcal F}

\newcommand{\abs}[1]{\left\lvert{#1}\right\rvert}

\def\path{semi-path}

\begin{document}
\title{The Structure of Hypergraphs without long Berge cycles}

\author[1,2]{Ervin Gy\H{o}ri} 
\author[3]{Nathan Lemons}
\author[1,2]{Nika Salia}
\author[2,4]{Oscar Zamora} 

\affil[1]{Alfr\'ed R\'enyi Institute of Mathematics, Hungarian Academy of Sciences. \newline
 \texttt{gyori.ervin@renyi.mta.hu, nika@renyi.hu}}
\affil[2]{Central European University, Budapest.     }
\affil[3] {Theoretical Division, Los Alamos National Laboratory}
 \affil[4]{Universidad de Costa Rica, San Jos\'e. \par
 \texttt{oscar.zamoraluna@ucr.ac.cr}}

\maketitle

\begin{abstract}
We study the structure of $r$-uniform hypergraphs containing no Berge cycles of length at least $k$ for $k \leq r$, and determine that such hypergraphs have some special substructure. In particular we determine the extremal number of such hypergraphs, giving an affirmative answer to the conjectured value when $k=r$ and giving a a simple solution to a recent result of Kostochka-Luo when $k < r$. 
\end{abstract}

\section{Introduction}
 
In 1959 Erd\H{o}s and Gallai proved the following results on the Tur\'an number of paths and families of long cycles.

\begin{theorem}[Erd\H{o}s, Gallai~\cite{Er-Ga}]\label{EG1} Let $n \geq k \geq 1$. If $G$ is an $n$-vertex graph that does not contain a path of length $k$, then $e(G) \leq \frac{(k - 1)n}{2}$.
\end{theorem}

\begin{theorem}[Erd\H{o}s, Gallai~\cite{Er-Ga}]\label{EG2} Let $n \geq k \geq 3$. If $G$ is an $n$-vertex graph that does not contain a cycle of length at least $k$, then $e(G) \leq \frac{(k - 1)(n - 1)}{2}$.
\end{theorem}

In fact, Theorem \ref{EG1} was deduced as a simple corollary of Theorem \ref{EG2}. Recently numerous mathematicians started investigating similar problems for $r$-uniform hypergraphs. We will refer to $r$-uniform hypergraphs as an $r$-graphs for simplicity. 
All $r$-graphs are simple (i.e. contain no multiple edges), unless stated otherwise.  
\begin{definition}
A \emph{Berge cycle} of length $t$ in a hypergraph, is an alternating sequence of distinct vertices and hyperedges,  $v_0, e_1, v_1, e_2, v_2, \dots, v_{t-1}, e_t,v_0$ such that, $v_{i-1},v_i \in e_i$, for $i =1,2,\dots t$, (where indices taken modulo $t$).
  \end{definition}


\begin{definition} 
A \emph{Berge path} of length $t$ in a hypergraph, is an alternating sequence of distinct vertices and hyperedges,  $v_0, e_1, v_1, e_2, v_2, \dots, e_t, v_t$ such that, $v_{i-1},v_i \in e_i$, for $i =1,2,\dots t$.  
\end{definition}

The first extension of Erd\H{o}s and Gallai~\cite{Er-Ga} result, was by Gy\H{o}ri, Katona, and Lemons~\cite{GyoKaLe}, who extended Theorem \ref{EG1} for $r$-graphs. It turns out that the extremal numbers have a different behavior when $k \leq r$ and $k > r$

\begin{theorem}[Gy\H{o}ri, Katona and Lemons~\cite{GyoKaLe}]\label{GKL1} Let $r \geq k \geq 3$, and let $\h$ be an $n$-vertex $r$-graph with no Berge path of length $k$. Then $e(\h) \leq \frac{(k-1)n}{r+1}.$
\end{theorem}

\begin{theorem}[Gy\H{o}ri, Katona and Lemons~\cite{GyoKaLe}]\label{GKL2}  Let $k > r + 1 > 3$, and let $\h$ be an $n$-vertex $r$-graph with no Berge-path of length $k$. Then $e(\h) \leq \frac{n}{k}{k \choose r}$. 
\end{theorem}

The remaining case when $k = r + 1$ was solved later by Davoodi, Gy\H{o}ri, Methuku, and Tompkins~\cite{DavoodiGMT}, the extremal number matches the upper bound of Theorem  \ref{GKL2}.

Similarly the extremal hypergraphs when Berge cycles of length at least $k$ are forbidden, are different in the cases when $k \geq r + 2$ and $k \leq r+1$ with an exceptional third case when $k=r$.
The latter has a surprisingly different extremal hypergraph. F\H{u}redi, Kostochka and Luo~\cite{furedi2018avoiding} provide sharp bounds and extremal constructions for infinitely many $n$, for $k\geq r+3\geq6$. Later they~\cite{furedi2018avoiding2} also determined exact bounds and extremal constructions for all $n$, for the case  $k \geq r+4$. Kostochka and Luo~\cite{KostochkaLuo} determine a bound for $k \leq r-1$ which is sharp for infinitely many $n$. Ergemlidze, Gy\H{o}ry, Metukhu, Salia, Tompikns and Zamora~\cite{ergemlidze2018avoiding} determine a bound in the cases where $k \in \{r+1,r+2\}$. The case when $k=r$ remained open.  Both papers~\cite{KostochkaLuo,ergemlidze2018avoiding} conjectured the maximum number of edges to be bounded by $\max\Big\{\frac{(n-1)(r-1)}{r},n-(r-1)\Big\}$ (See Figure \ref{F1}). 

\begin{theorem}[F\"uredi, Kostochka and Luo~\cite{furedi2018avoiding,furedi2018avoiding2}] Let $r \geq 3$ and $k \geq r + 3$, and suppose $\h$ is an $n$-vertex $r$-graph with no Berge cycle of length $k$ or longer. Then $e(\h) \leq \frac{n-1}{k-2}{k-1 \choose r}.$
\end{theorem}

\begin{theorem}[Ergemlidze et al.~\cite{ergemlidze2018avoiding}] If $k \geq 4$ and $\h$ is an $n$-vertex $r$-graph with no Berge cycles of length at least $k$, then  if $k = r + 1$ then $e(\h) \leq n - 1$, and if $k = r + 2$ then $e(\h) \leq \frac{(n-1)(r+1)}{r}$.
\end{theorem}

\begin{theorem}[Kostochka, Luo \cite{KostochkaLuo}]
\label{KL}
Let $k \geq 4, r \geq k + 1$ and let $\h$ be an $n$-vertex $r$-uniform multi-hypergraph, each edge of $\h$ has multiplicity at most $k - 2$. If $\h$ has no Berge-cycles of length at least $k$,
then $e(H) \leq \frac{(k-1)(n-1)}{r}$. 
\end{theorem}

Kostochka and Luo obtain their result from the incidence bipartite graph by investigating the structure of 2-connected bipartite graphs. In a similar way a previous result of Jackson~\cite{jackson1981cycles} gives an upper bound on the number of edges of a multi $r$-graph with no Berge cycle of length at least $r$.
\begin{theorem}[Jackson~\cite{jackson1981cycles}]
\label{bipartite}
Let $G$ be a bipartite graph with bipartition $A$ and $B$ such that $\abs{A} = n$ and every vertex in $B$ has degree at least $r$, if $\abs{B} > \floor{\frac{n-1}{r-1}}(r-1)$ then $G$ contains a cycle of length at least $2r$. 
\end{theorem}


In this paper we study the structure of $r$-graphs containing no Berge cycles of length at least $k$, for all $3 \leq k \leq r$. By exploring the structure of the hypergraphs, instead of bipartite graphs, we are able to find extremal number in the case when $k=r$, which also gives us a simple proof for Theorem \ref{KL}. Even more our method lets us determine the extremal number for every value of $n$  in both simple $r$-graphs and multi $r$-graphs.




\section{Notation and results}

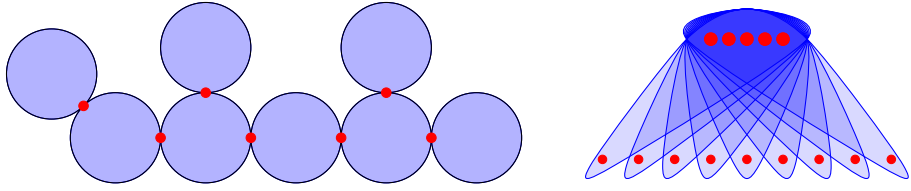
\begin{figure}
 \centering
\begin{tikzpicture}[scale = 0.6] 
\filldraw[blue,fill opacity=0.3] (1,0) arc (0:360: 1cm and 1cm);
\draw (1,0) arc (0:360: 1cm and 1cm);

\filldraw[blue,fill opacity=0.3] (1,0) + (xyz polar cs:angle=135,radius=2) arc (0:360: 1cm and 1cm);
\draw (1,0) + (xyz polar cs:angle=135,radius=2) arc (0:360: 1cm and 1cm);

\filldraw[blue,fill opacity=0.3] (3,0) arc (0:360: 1cm and 1cm);
\draw (3,0) arc (0:360: 1cm and 1cm);

\filldraw[blue,fill opacity=0.3] (5,0) arc (0:360: 1cm and 1cm);
\draw (5,0) arc (0:360: 1cm and 1cm);

\filldraw[blue,fill opacity=0.3] (7,0) arc (0:360: 1cm and 1cm);
\draw (7,0) arc (0:360: 1cm and 1cm);

\filldraw[blue,fill opacity=0.3] (9,0) arc (0:360: 1cm and 1cm);
\draw (9,0) arc (0:360: 1cm and 1cm);

\filldraw[blue,fill opacity=0.3] (3,2) arc (0:360: 1cm and 1cm);
\draw (3,2) arc (0:360: 1cm and 1cm);

\filldraw[blue,fill opacity=0.3] (7,2) arc (0:360: 1cm and 1cm);
\draw (7,2) arc (0:360: 1cm and 1cm);

\filldraw[red] (1,0) circle (3pt) (3,0) circle (3pt) (5,0) circle (3pt) (7,0) circle (3pt)
(2,1) circle (3pt) (6,1) circle (3pt)  (xyz polar cs:angle=135,radius=1) circle(3pt);

\end{tikzpicture}\qquad
\begin{tikzpicture}[scale=.8, rotate = -90]

\foreach \x in {-2.4,-1.8,...,2.4}{
\filldraw[blue, fill opacity=0.15] plot [smooth cycle] coordinates {(0,1) (-0.5,0) (0,-1) (2.1,\x-0.2) (2.1,\x+0.2) } ;
\filldraw[red] (2,\x) circle (2pt);
}

\filldraw[red] (0,.6) circle (3pt)  (0,.3) circle (3pt) (0,.0) circle (3pt) (0,-.3) circle (3pt) (0,-.6) circle (3pt);

\end{tikzpicture}
\caption{The extremal graphs from Theorems \ref{k<r}, \ref{k=r} and \ref{multi}. The figure on the left is a block tree, each block contains same number of vertices,  either $r$ in the case of multi-hypergraphs or $r+1$ otherwise and  $k-1$ hyperedges. The figure on the right is $\S_{n}^{(r)}$ the $n$-vertex $r$-star, each hyperedge share the same $r-1$ vertices.} \label{F1}
\end{figure}

Given a hypergraph $\h$, let $V(\h)$ and $E(\h)$ denote the set of vertices and hyperedges of $\h$, respectively, and let $v(\h) := \abs{V(\h)}$, $e(\h) := \abs{E(\h)}$. We denote by $\mathbb{1}_{r \mathbb{N}^*}(n)$, the characteristic function of $r \mathbb{N}^*$: the function which is 1 when $n$ is a positive multiple of $r$ and 0 otherwise. 
 A hypergraph is $\F$-free if it doesn't contain a copy of any hypergraph from the family $\F$ as a sub-hypergraph.  In the following, we are particularly interested in the families $\BP_k$ and $\BC_{\geq k}$, the family of Berge path of length $k$ and the family of Berge cycles of length at least $k$, respectively.
The Tur\'an number $ex_r(n,\F)$ and $ex_r^{multi}(n,\F)$  are the maximum number of hyperedges in a $\F$-free hypergraph or multi-hypergraph respectively on $n$ vertices.

Let $\mathcal H$ be a hypergraph. Then its \emph{2-shadow}, denoted by $\partial_2 \h$, is the collection of pairs of vertices that lie in some hyperedge of $\h$.
%
The graph $\h$ is \emph{connected} if $\partial_2(\h)$ is a connected graph.


Let $n,k,r$ be integers such that $k \leq r$, for fix $s \in\{r,r+1\}$. 
A $r$-graph $\h$ is called a $(s,k-1)$-block tree if $\partial_2(\h)$ is connected and every 
$2$-connected block
 of $\partial_2(\h)$ consists of $s$ vertices which induce $k-1$ hyperedges in $\h$. 
A $(s, k - 1)$-block tree  contains no Berge-cycle of length at least $k$, because each of its blocks contain fewer than $k$ hyperedges, see Figure \ref{F1}.

We define the $r$-star, $\S_n^{(r)}$, as the $n$-vertex $r$-graph with vertex set $V(\S_n^{(r)}) = \{v_1,v_2,\dots,v_n\}$ and edge set $E(\S_n^{(r)}) = \{\{v_1,v_2,\dots,v_{r-1},v_i\}: r \leq i \leq n\}$, the set $\{v_1,v_2,\dots,v_{r-1}\}$ is called the center of the star.
Since $\S_n^{(r)}$ has just $r-1$ vertices of degree bigger than 1, then $\S_n^{(r)}$ contains no Berge cycle of length at least $r$.

\begin{definition}
For a set $S\subseteq V$, the \emph{hyperedge neighborhood} of $S$ in a $r$-graph $\h$ is the set \begin{displaymath}
N_h(S):=\{h\in E(\h) | h \cap \S \not=\emptyset \}
\end{displaymath} of hyperedges that are incident with at least one vertex of $S$.
\end{definition}

Our Main results are: 



\begin{theorem}\label{k<r} Let $k,n$ and $r$ be positive integers such that $4\leq k< r$, then 
\begin{displaymath}  \ex_r(n,\BC_{\geq k})= \floor{\frac{n-1}{r}}(k-1)+ \mathbb{1}_{r \mathbb{N}^*}(n) \end{displaymath}
If $r|(n-1)$ the only extremal $n$-vertex $r$-graphs are the $(r+1,k-1)$-block trees.
\end{theorem}

We note that as a corollary of Theorem $\ref{k<r}$ we obtain a slightly stronger version of Theorem \ref{GKL1}

\begin{corollary}\label{P1} Let $k,n$ and $r$ be positive integer with $4\leq k \leq r$, then \begin{displaymath}
\ex_r(n,\BP_{k}) = \floor{\frac{n}{r+1}}(k-1)+ \mathbb{1}_{(r+1) \mathbb{N}^*}(n+1)
\end{displaymath}
\end{corollary}



\begin{theorem}\label{k=r}
Let $r>2$ and $n$ be positive integers, then
\begin{displaymath}
\ex_r(n,\BC_{\geq r})= \max\bigg\{\floor{\frac{n-1}{r}}(r-1), n-r+1\bigg\} \end{displaymath}
When $n-r+1 > \frac{n-1}{r}(r-1)$ the only extremal graph is $\S_n^{(r)}$.  When $\frac{n-1}{r}(r-1) > n-r+1$ and $r|(n-1)$ the only extremal graphs are the $(r+1,k-1)$-block trees.
\end{theorem}



\begin{remark} In particular when $n\geq r(r-2) +2$, we have that $\ex_{r}(n,\BC_{\geq r}) = n-r+1$ and $\S_n^{(r)}$ is the only extremal hypergraph.
\end{remark}

\begin{theorem}\label{multi} Let $k,n$ and $r$ be positive integers such that $2\leq k\leq r$.  
Then 
\begin{displaymath}  \ex_r^{multi}(n,\BC_{\geq k})= \floor{\frac{n-1}{r-1}}(k-1) \end{displaymath}
If $r-1|(n-1)$ the only extremal graphs with $n$ vertices are the $(r,k-1)$-block trees.
\end{theorem}

As a corollary of Theorem \ref{multi} we obtain a version of Theorem \ref{GKL1} with multiple hyperedges

\begin{corollary}\label{P2}  Let $k,n$ and $r$ be positive integer with $2\leq k \leq r$ then \begin{displaymath}
\ex_r^{multi}(n,\BP_{k}) = \floor{\frac{n}{r}}(k-1).
\end{displaymath} 
\end{corollary}

In fact all these results have essentially the same proof since, these results follow from our Lemma \ref{main}, which to some extent lets us understand the structure of long Berge cycle free hypergraphs.

\begin{lemma}
\label{main}
 Let $r,n$ and $m$ be positive integers, with $n> r$, and let $\h$  be an $n$-vertex $r$-graph 
 which is $\BC_{\geq k}$-free such that every hyperedge has multiplicity at most $m$.
 Then at least one of the following holds.
 \begin{itemize}
           
            \item [i)] There exists $S\subseteq V$ of size $r-1$ such that $\abs{N_{h}(S)} \leq m.$ Moreover, if $m < k-1$ there exists a set $S$ of size $r-1$ such that $N_{h}(S)$ is   $d\leq m$ copies of a hyperedge $h$ and $S\subset h$.

     \item [ii)] There exists $S\subseteq V$ of size $r$ such that $\abs{N_{h}(S)} \leq k-1.$ 
 

\item [iii)] $k =r$, $m < k-1$, and there exists $e \in E(\h)$ such that after removing $e$ from $\h$ the resulting $r$-graph can be decomposed in two $r$-graphs, $\S$ and $\K$ sharing one vertex, such that $\S$ is a $r$-star with at least $r-1$ edges, the shared vertex is in the center of $\S$, $e\cap V(\S)$ is a subset of the center of $\S$ and $v(\K) \geq 2.$
 \end{itemize}

In particular, since no hyperedge can have multiplicity larger than $k-1$, by setting $m = k-1$ we have that there exists a set $S$ of size $r-1$ incident with at most $k-1$ edges. 

 \end{lemma}

In Section \ref{proofs} we deduce Theorems \ref{k<r}, \ref{k=r} and \ref{multi} from Lemma \ref{main}, as well as their corollaries. We leave the proof of Lemma \ref{main} for Section  \ref{lemma}.


\section{Proof of main results}
\label{proofs}
This section contains two subsections in the first one we prove our main results Theorem \ref{k<r}, Theorem  \ref{k=r} and  Theorem \ref{multi} using Lemma \ref{main}. In the second subsection we prove   their corollaries.

\subsection{Proof of Theorem \ref{k<r}, \ref{k=r} and \ref{multi}}
To obtain the extremal constructions in Theorem \ref{k<r}, first we are going to show that in a $(r+1,k-1)$-block tree for every pair of vertices there exists a Berge path of length $k-1$ joining them, for this we prove the following statement by induction.

\begin{claim}\label{path} Let $r\geq 3$ and $\h_1$ a  multi (not necessarily uniform) hypergraph such that $v(\h_1) = r+1 > e(\h_1) \geq 2$, and every hyperedge $h \in E(\h_1)$, $h \not=V(\h_1)$ has size at least $r$ and multiplicity at most one. Then every pair of vertices of $\h_1$ are join by a Berge path of length $e(\h_1)$. 
\end{claim}

\begin{proof}
The proof is by induction on $r$. The case where $r = 3$ is simple to check, as well as the case when $e(\h_1) = 2$, since every edge contain all but at most one vertex. So suppose $r > 3$ and $e(\h_1) > 2$. 
Let $v,u$ be to distinct vertices, take any hyperedge $h$ containing $v$, then choose $w\in \h \backslash \{v,u\}$, consider $\h_2$ obtain by removing $v$ and $h$ from $\h_1$ and by deleting $v$ from the remaining hyperedges, then $\h_2$ satisfy the conditions of the claim, hence there exists a Berge path of length $e(\h_2) = e(\h_1)-1$ joining $w$ and $u$, we can extend this path with $h$ to be a Berge path of length $e(\h_1)$ joining $v$ and $u$.
\end{proof}

Therefore we proved that in a  $(r+1,k-1)$-block tree for every pair of vertices from the same block there exists a Berge path of length $k-1$ joining them hence the statement trivially holds for every pair of vertices too since $(r+1,k-1)$-block tree is connected hypergraph.
\begin{proof}[Proof of Theorem \ref{k<r}]
For the lower bound we can observe that a $(r+1,k-1)$-block tree on $ar+1$ vertices is a $\BC_{\geq k}$-free graph with $a(k-1)$ edges, for $n\in\{ar+1,ar+2,\dots,(a+1)r-1\}$ this proves the lower bound, if $n = (a+1)r$ add an extra edge containing $r-1$ new vertices to this construction and we will get a desired lower bound.

For the upper bound, let $\h$ is an $r$-uniform, $n$-vertex,  hypergraph,   without a Berge cycle of length at least $k$. The proof is by induction on the number of vertices. 
The theorem trivially holds for $n\leq r$. So suppose  $n>r$ and that the theorem holds for any graph with less than $n$ vertices, by Lemma \ref{main} there exists a set $S \subseteq V$ such that either $\abs{S} = r-1$ and $\abs{N_h(S)} = 1$ or $\abs{S} = r$ and $\abs{N_h(S)} = k-1$. Let $\h'$ be the graph induce by $V'= V\backslash S$.  Then either 
\begin{displaymath} e(\h) \leq 1  + e(\h')  \leq 1 + \floor{\frac{n-r}{r}}(k-1)+\mathbb{1}_{r \mathbb{N}^*}(n-r+1) \leq \floor{\frac{n-1}{r}}(k-1)+\cdot\mathbb{1}_{r \mathbb{N}^*}(n), \mbox{ or}\end{displaymath} 
\begin{displaymath} e(\h) \leq (k-1) + e(\h') \leq (k-1) + \floor{\frac{n-(r-1)-1}{r}}(k-1) + \mathbb{1}_{r \mathbb{N}^*}(n-r) = \floor{\frac{n-1}{r}}(k-1) + \mathbb{1}_{r \mathbb{N}^*}(n).
\end{displaymath}


From the above calculations equality holds, only when $\abs{S}=r$, $\abs{N_h(S)}=k-1$ and $e(\h')=\floor{\frac{n-r-1}{r}}(k-1)$ or  $r|n$, $\abs{S}=r-1$, $\abs{N_h(S)}=1$ and $e(\h') = \floor{\frac{n-r-1}{r}}(k-1) +1$.
If $r|n-1$, we prove the only extremal hypergraph is a  $(r+1,k-1)$-block tree. We have $\abs{S} = r$ and by induction $\h'$ is a $(r+1,k-1)$-block tree. For any hyperedge $h$ incident with $S$ we have that $\abs{h \cap V'} \leq 1$, otherwise we have a Berge cycle of length at least $k$ in $\h$, a contradiction, since any two vertices of $V'$ are joined by a $k-1$ length Berge path in $\h'$. If there exists two hyperedges $h$, $h'$ incident with $S$ such that, there exists two distinct vertices $v$, $v' \in V'$, such that  $v \in h$ and $v' \in h'$ then both $h\backslash\{v\}$ and $h'\backslash\{v'\}$ have $r-1$ elements in $S$, then these hyperedges must intersect in a vertex $x$, $x\in S$. So $v, h, x, h', v'$ together with a Berge path of length at least $k-1$  joining $v$ to $v'$ in $\h'$ is a Berge cycle of length at least $k+1$. Therefore every edge in  $N_h(S)$ is either $S$ or intersect the same vertex $v$ of $V'$, hence $\h$ is a $(r,k-1)$-block tree.
\end{proof} 


\begin{remark} If $\h$ is a $n$-vertex multi $r$-graph in which each edge has multiplicity at most $m\leq k-1$ and contains no Berge cycle of length at least $k$, then Lemma \ref{main} implies  $e(\h) \leq \max\{a(k-1) + bm:ar + b(r-1) < n\},$ this holds for all $k\geq2$. 
\end{remark}

\begin{proof}[Proof of Theorem \ref{multi}] This theorem follows by induction in the same way as Theorem \ref{k<r} since we can always find a set $S$ of size $r-1$ incident with at most $k-1$ edges.
\end{proof}

\begin{proof}[Proof of Theorems \ref{k=r}]

We will assume by induction that Theorem  \ref{k=r} holds for $n'<n$. Note that for $n'>2$, $e(\h) \leq n'-2$ and equality holds only for $\S_{n'}^{(3)}$ when $r=3$, or a $(r+1,r-1)$-block when $n' = r+1$, in particular equality only holds for connected hypergraphs. Applying Lemma \ref{main}, one of  $i)$,  $ii)$ or  $iii)$ must hold.

If $iii)$ holds in Lemma \ref{main}, let $\S$ and $\K$ be the given decomposition after removing the hyperedge $e$. 
Let $v$ be the only vertex in $V(\S) \cap V(\K)$, $u \in e \cap V(\S)$ and $w \in e\cap(\K)$, with both $u,w$ different from $v$. If $v(\K) \geq 3$, then $e(\h) = e(\S) + e(\K) + 1 \leq v(\S)-(r-1)+v(\K)-2+1 = n-(r-1)$, but equality is not possible, since by connectivity of $\K$ there is a Berge path from $w$ to $v$ in $\K$ and we have a Berge path of length $r-2$ in $S$ from $v$ to $u$, finally we can use the hyperedge $e$ to connect $u$ to $w$, we get a Berge path of length at least $k$, a contradiction. So $\h$ has $n-(r-1)$ edges only if $v(\K) = 2$, therefore $e$ contains the center of $S$ and the only vertex of $V(\K)\backslash \{v\}$, hence $\h=\S_n^{(r)}$. Finally we have $e(\h) \leq n-(r-1)$ and equality holds when $\h =\S_n^{(r)}$.


If $\floor{\frac{n-1}{r}}(r-1) \geq n-r+1$ then either $iii)$ in Lemma \ref{main} holds and $e(\h)\leq n-(r-1)$, or the  proof of extremal number follows by induction in the similar way as Theorem \ref{k<r}.

If $n-r+1 > \floor{\frac{n-1}{r}}(r-1)$, and $i)$  holds in Lemma \ref{main} then we have $e(\h) < n-(r-1)$ since $n'-r+1 \geq \floor{\frac{n'-1}{r}}(r-1)$ for $n'=n-(r-1)$, a contradiction. If $n-r+1 > \floor{\frac{n-1}{r}}(r-1)$, and $ii)$  holds in Lemma \ref{main} then we have $e(\h) < n-(r-1)$ since $n'-r+1 \geq \floor{\frac{n'-1}{r}}(r-1)$ for $n' = n-r$, a contradiction.

If  $n-r+1 > \floor{\frac{n-1}{r}}(r-1)$  then $iii)$ should hold in Lemma \ref{main}, hence  $\h =\S_n^{(r)}$. 

Suppose $n-r+1>\floor{\frac{n-1}{r}}(r-1)$. If $i)$  holds in Lemma \ref{main} then we have $e(\h) < n-(r-1)$ since $n'-r+1 \geq \floor{\frac{n'-1}{r}}(r-1)$ for $n'=n-(r-1)$, a contradiction.  If $ii)$  holds in Lemma \ref{main} then we have $e(\h) < n-(r-1)$ since $n'-r+1 \geq \floor{\frac{n'-1}{r}}(r-1)$ for $n' = n-r$, which is also a contradiction. Therefore $iii)$  holds in Lemma \ref{main}, hence  $\h =\S_n^{(r)}$. 
\end{proof}

\subsection{Proof of Corollaries \ref{P1} and \ref{P2}}

\begin{proof}[Proof of Corollary \ref{P1}]
Let $\h$ be an $n$-vertex $r$-graph containing no Berge path of length $k$. Define a $(r+1)$-graph $\h'$ by adding a new vertex $v$ to the vertex set of $\h$ and extending every hyperedge of $\h$ with $v$.

If $\h'$ is $\BC_{\geq k}$-free, then from Theorem \ref{k<r}, we have 

\begin{displaymath}
e(\h) = e(\h') \leq \floor{\frac{n+1-1}{r+1}}(k-1) + \mathbb{1}_{(r+1) \mathbb{N}^*}(n+1)= \floor{\frac{n}{r}}(k-1) + \mathbb{1}_{(r+1) \mathbb{N}^*}(n+1).
\end{displaymath}

If $\h'$ contains a copy of a Berge cycle   $v_1,h_1,v_2\dots,h_{\ell-1},v_\ell, h_{\ell}, v_1$, of length $\ell$, for some $\ell \geq k$. If $v$ is one of the defining vertices, suppose without loss of generality $v=v_1$, and let $h_i'= h_i\backslash\{v\}$ for each $i=1,2\dots\ell$ then $\abs{(h_1' \cup h_k') \backslash \{v_2,\dots, v_{k}\}} \geq r+1-(k-1) \geq 2$ and that set intersects both $h'_1$ and $h'_k$ hyperedges. Therefore we can find two distinct vertices $u \in h'_1$ and $u' \in h'_k$ different from all $v_i$, $i \in \{1,2,\dots, k\}$ then $u,h'_1 ,v_2,h'_2,v_3,\dots,h'_{k-1},v_k, h'_k, u'$ is a Berge path of length $k$ in $\h$, a contradiction. If $v$ is not one of the defining vertices, then similar argument leads us to contradiction.

\end{proof}

\begin{proof}[Proof of Corollary \ref{P2}] This follows in a similar way as the previous corollary, by constructing a $\BC_{\geq k}$-free $r$-multi-graph $\h'$.\\
Hence, by Theorem \ref{multi},  $\displaystyle e(\h) = e(H') \leq \floor{\frac{n+1-1}{r+1-1}}(k-1) = \floor{\frac{n}{r}}(k-1).$
\end{proof}
\section{Proof of  Lemma \ref{main}}\label{lemma}



\begin{definition} 
A \emph{semi-path} of length $t$ in a hypergraph, is an alternating sequence of distinct hyperedges and vertices,  $e_1, v_1, e_2, v_2, \dots, e_t, v_t$ (starting with a hyperedge and ending in a vertex) such that, $v_1 \in e_1$ and $v_{i-1},v_i \in e_i$, for $i =2,3,\dots t$.  
\end{definition}


Let $r\geq k \geq 3$ be fix integers and let $\h$ be a $\BC_{\geq k}$-free  multi $r$-graph,  
consider a  semi-path  $P =  e_1, v_1, e_2, v_2, \dots, e_t, v_t$ of maximal length. Consider $P'$ the semi-path $e_1v_1e_2,v_2,\dots,e_\ell,v_t$ obtained from the first $\ell$ vertices and hyperedges of $P$, where $\ell = \min\{k-1,t\}$, let $\F=\{e_1, e_2, \dots, e_l\}$ and $U=\{v_1, v_2, \dots, v_l\}$, the defining vertices and hyperedges of this path. 
Note that  $|e_1 \cap U| \leq k-1 < r$, so  $e_1 \backslash U \not = \emptyset$.

First we will show that any vertex, from $e_1 \cap U$, is only incident with the defining hyperedges in $\F$.

\begin{lemma}\label{first hyperedge}
Suppose $w \in e_1 \backslash U$, then $N_h(w) \subseteq \F$. Hence $N_h(e_1\backslash U) \subseteq \F$.
\end{lemma}
\begin{proof}
If $w$ is incident with a hyperedge of $P$ not in $\F$, let $j\geq k$ be the smallest index such that $w$ is incident with $e_j$, then $v_1,e_2,v_2\dots,v_{j-1},e_j,w,e_1,v_1$ is a Berge cycle of length at least $k$, a contradiction, 
If $w$ is incident with an edge $e$ not in the Berge path $P$, then $e,w,P$ is a longer semi-path, a contradiction to the maximality of $P$.

For simplicity Lemma \ref{first hyperedge} was stated and proved for a maximal semi-path $\P$, but it similarly holds for every maximal semi-path. Hence we may apply Lemma \ref{first hyperedge} for other maximal semi-paths.
\end{proof}


For each defining vertex $v_i$, $v_i\in e_1\cap U$, we find another maximal semi-path by rearranging $P$, starting at $e_i$, without changing  the set of the first $\ell$ vertices and hyperedges.

\begin{lemma}\label{missing}
If for some $i$ we have that $v_i\in e_1\cap U$,  then $N_h(e_{i}\backslash U) \subseteq  \F$.
\end{lemma}
\begin{proof}
Consider the semi-path $e_i,v_{i-1},e_{i-1},v_{i-2},\dots,e_2,v_1,e_1,v_i,e_{i+1},v_{i+1},\dots,e_t,v_t$, this semi-path has length $t$, so it is maximal, then $N_h(e_{i}\backslash U) \subseteq  \F$ follows from Lemma \ref{first hyperedge} for this path.
\end{proof}

\begin{lemma} \label{extra} If there are two vertices $v_i,v_j \in e_1 \cap U$, with $i > j$ such that $(e_i \cap e_j) \backslash U \not = \emptyset,$ then  $N_h(v_{i-1})\subseteq \F$ and $N_h(v_{j})\subseteq \F$. 
\end{lemma}

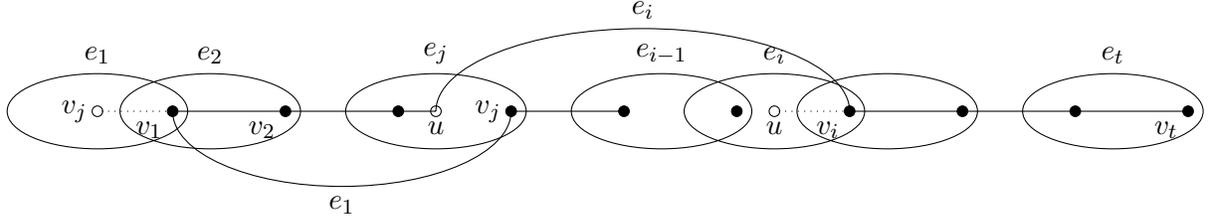
\begin{figure}[t]
\begin{center}
\begin{tikzpicture}[scale=1]


\draw (0,0) arc(90:450:1.2cm and 0.5cm) (1.5,0) arc(90:450:1.2cm and 0.5cm)  (4.5,0) arc(90:450:1.2cm and 0.5cm)  (7.5,0) arc(90:450:1.2cm and 0.5cm) (9,0) arc(90:450:1.2cm and 0.5cm) (10.5,0) arc(90:450:1.2cm and 0.5cm) (13.5,0) arc(90:450:1.2cm and 0.5cm);

\draw (0,0) node[above]{$e_1$} (1.5,0) node[above]{$e_2$} (3,0) node[above]{$ $} (4.5,0) node[above]{$e_j$}   (7.5,0) node[above]{$e_{i-1}$} (9,0) node[above]{$e_i$} (13.5,0) node[above]{$e_t$};

\filldraw (1,-.5) circle (2pt) node[below left]{$v_1$} (2.5,-.5) circle (2pt) node[below left]{$v_2$} (4,-.5) circle (2pt) (5.5,-.5) circle (2pt) node[left]{$v_j$} (7,-.5) circle (2pt) (8.5,-.5) circle (2pt) (10,-.5) circle (2pt) node[below left]{$v_i$} (11.5,-.5) circle (2pt) (13,-.5) circle (2pt) (14.5,-.5) circle (2pt) node[below left]{$v_t$};

\draw (0,-.5) circle (2pt) node[left]{$v_j$}  (4.5,-.5) circle (2pt) node[below]{$u$} (9,-.5) circle (2pt) node[below]{$u$};



\draw (14.5,-.5) -- (10,-.5) arc(0:180:2.75 cm and 1.1cm)   (4.5,-.5) -- (1,-.5) arc(-180:0: 2.25cm and 1cm)   (5.5,-.5)-- (7,-.5);
\draw[dotted] (10,-.5) -- (9,-.5)  (1,-.5) -- (0,-.5);

\draw (3.25,-1.5) node[below]{$e_1$};
\draw (7.25,0.6) node[above]{$e_i$};

\end{tikzpicture}
\caption{Semi-path $P_1$ in the proof of Lemma \ref{extra}}
\label{spath}
\end{center}
\end{figure}

\begin{proof}
Fix $u \in (e_i \cap e_j) \backslash U$ and consider the following maximal length semi-paths (see Figure \ref{spath})\\
$P_1 = e_{i-1},v_{i-2},e_{i-2},v_{i-3},\dots,e_{j+1},v_{j},e_1v_1,e_2,v_2,\dots,v_{j-1},e_j,u,e_{i},v_i,e_{i+1},v_{i+1},\dots,e_{t},v_t,$ and \\
$P_2 = e_{j+1},v_{j+1},e_{j+2},v_{j+2},\dots,v_{i-1},e_{i},u,e_{j}v_{j-1},e_{j-1},v_{j-2},\dots,e_2,v_1,e_1,v_i,e_{i+1},v_{i+1},\dots,e_t,v_t.$ 
Applying Lemma \ref{first hyperedge} for a maximal semi-path $P_1$ and $P_2$, we get $N_h(v_{i-1})\subseteq \F$ and $N_h(v_{j+1})\subseteq \F$. \end{proof}


Let $d \leq m$ be an integer such that $e_1 = e_2 = \cdots = e_d \not= e_{d+1}$. 
\begin{claim}
\label{simple}
If $e_1 \cap U = \{v_1,v_2,v_3,\dots,v_d\}$ then either $e_1\backslash \{v_d\}$ is incident with $d$,  $d\leq m$, hyperedges or there exists a set $S$ of size $r$ such that $N_h(S) \subseteq \F$. In particular if $e_1 \cap U = \{v_1,v_2,v_3,\dots,v_d\}$ then Lemma \ref{main} holds too.
\end{claim}

\begin{proof}
First note that the vertices $v_1,v_2,\dots,v_{d-1}$ can be exchanged with the vertices of $e_1 \backslash U \not= \emptyset$, hence from the Lemma \ref{first hyperedge}, we have $N_{h}(e_1\backslash \{v_d\}) \subseteq \F$.
Suppose $w \in e_1 \backslash \{v_d\}$ is incident with a hyperedge $e_j$, $l\geq j>d$, we may assume $w \in e_1 \backslash U$, then the semi-path $P' = e_{j-1},v_{j-2},e_{j-2},v_{j-3},\dots$ $\dots,e_2,v_1,e_1,w,e_j,v_j,\dots,e_t,v_t$ with the maximal length. Since $v_{j-1}$ is a non defining vertex in the first hyperedge of a maximal semi-path $P'$, applying Lemma \ref{first hyperedge} to $P'$,  we have that $N_h(v_{j-1}) \subseteq \F$, therefore the set $(v_1\backslash \{v_d\}) \cup \{v_{j-1}\}$ is a set of  $r$  vertices incident with at most $k-1$ hyperedges from $\F$. Otherwise, if there is no such $w$ then we have a set of $r-1$ vertices, $e_1 \backslash\{v_d\}$ incident with at most $m$ hyperedges.
\end{proof}
From here we may assume that $\abs{e_1 \cap U} > d$. Let $e_1 \cap U = \{v_{i_0},v_{i_1},v_{i_2},\dots,v_{i_s}\}$, where $1=i_0<i_1 < i_2 < \cdots < i_s$, define  recursively the sets $A_1:= e_1 \backslash U$ and for $j=1,\dots,s$, if $(e_{i_j}\backslash U) \cap A_j = \emptyset$, take $A_{j+1} := A_{j} \cup (e_{i_j}\backslash U)$, otherwise take $A_{j+1} := A_{j} \cup (e_{i_j}\backslash U) \cup \{v_{i_j-1}\}$, let $A :=A_{s+1}$. 
Note that $\abs{A_j}<\abs{A_{j+1}}$, for all $j\in \{1,2,\dots, s\}$, so $\abs{A} \geq |A_1| + s \geq r-1$,  by Lemmas \ref{first hyperedge}, \ref{missing} and \ref{extra},  we  have that $N_h(A) \subseteq \F$.
If $m = k-1$ then $A$ is a set of at least $r-1$ vertices incident with at most $m$ hyperedges, hence Lemma \ref{main} holds.  If $\abs{A} \geq r$  then $A$ is a set of at least $r$ vertices incident with at most $k-1$ hyperedges, hence Lemma \ref{main} holds. From here we may assume $m < k-1$ and $\abs{A} = r-1$. Observe that $\abs{A} = r-1$ is only possible if for every $i=1,2,\dots,s$, $\abs{A_{i+1}}=\abs{A_{i}}+1$. We will assume, without loss of generality, that among all possible semi-paths of maximal length, $P$ is a one for which  $\abs{e_1 \backslash U}$ is minimal. There are two cases:

\begin{itemize}

\item[Case 1:] There exists an index $j \geq d$, such that $A_j$ intersects $(e_{i_j}\backslash U)$, let $j'$ be the first such index, then there is another index $d-1 \leq q< j'$ such that $ (e_{i_j}\backslash U) \cap (e_{i_{q}}\backslash U) \not= \emptyset$, and let $u$ be an element in the intersection. 

If $i_{q} < i_{j'}-1$ then $v_{i_{q}} \not \in A$ from minimality of $j'$, and by Lemma \ref{extra}, $N_h(v_{i_{q}}) \subseteq \F$ so $A \cup \{v_{i_{q}}\}$ is a set of vertices, of size $r$, incident with at most $k-1$ hyperedges, hence Lemma \ref{main} holds. 

If $d < i_{q} = i_{j'}-1 $ then by applying Lemma \ref{first hyperedge} to a maximal semi-path \\
\begin{displaymath}
e_{i_{q}-1},v_{i_{q}-2},e_{i_{q}-2},v_{i_{q}-3},\dots,v_2,e_2,v_1,e_1,v_{i_{q}},e_{i_{q}},u,e_{i_{q}+1},v_{i_{q}+1},e_{i_{q}+2},\dots,e_t,v_t
\end{displaymath}
we get $N_h(v_{i_{q}-1}) \subseteq \F$, since $v_{i_{q}-1}$ is a non defining vertex in the first hyperedge. Also we have  $v_{i_{q}-1} \not \in A$ from mentality of $j'$, hence $A \cup \{v_{i_{q}-1}\}$ is a set of  $r$ vertices, incident with at most $k-1$ hyperedges and therefore \ref{main} holds. 

If $d=i_{q}=i_{{j'}}-1$, note that this implies that $(e_{d+1}\backslash U) \subseteq (e_1\backslash U),$  otherwise $A_{d+1}$ would have at least two new elements, but by minimality of $\abs{e_1\backslash U}$, we have   $(e_{d+1}\backslash U) = (e_1\backslash U)$. Fix any vertex $v_x$, $v_x \in U \cap (e_{d+1}\backslash e_1)$. We need a similar lemma as Lemma \ref{extra}.

\begin{claim}\label{special}
Suppose $v_{j} \in e_1$ is such that $(e_j\backslash U)$ intersects $(e_x\backslash U)$ then $N_{h}(v_{\max\{j,x\}-1}) \subseteq F.$
\end{claim} 
We  skip the proof of Lemma \ref{special}, since it is similar to the proof of Lemma \ref{extra}. 
%
Let $(e_1 \cap U) \cup \{v_x\} = \{v_{j_0},v_{j_1},\dots,v_{j_{s+1}}\}$, where $1=j_0 < j_1 < \cdots < j_{s+1}$,   define recursively the following sets $B_1 = e_1 \backslash U,$ and for $c=1,2,\dots,s+1$ let $B_{c+1} = B_{c} \cup (e_{v_{j_c}}\backslash U)$, if $B_{c} \cap (e_{v_{j_c}}\backslash U) = \emptyset$, otherwise take $B_{c+1} = B_c \cup \{v_{j_c}\}.$ Finally  $B_{s+2}$ has size at least $r$ and  is incident with at most $k-1$ hyperedges, therefore Lemma \ref{main} holds.
    
    \item[Case 2:] For every index $j \geq d$, $A_j$ and $(e_{i_j}\backslash U)$ are disjoint. We have that $r-1 = \abs{A} = \abs{e_1\backslash U} + (d-1) + \abs{(e_{i_d}\backslash U)} + \cdots + \abs{(e_{i_s}\backslash U)},$
    this implies that $\abs{(e_{i_j}\backslash U)} = 1$ for every $j$, hence $ \abs{U} \geq r-1$, but since $k-1 \geq \abs{U}$, we have that $k = r$ and $\abs{U}=r-1$. 
   So there exists distinct vertices $u_d,u_{d+1},\dots,u_{r-1}$ such that $e_i = U \cup\{u_i\}$ for each  $i\in \{d,d+1,\dots, r-1\}$ and $A = \{v_1,v_2,\dots,v_{d-1},u_d,u_{d+1},\dots,u_{r-1}\}$. If $d>1$, take a maximal semi-path, which  we get by exchanging $v_{r-2}$ with $v_1$ and $v_1$ with $u_d$, in a semi-path $P$ and apply Lemma \ref{first hyperedge}, we get $N_h(v_{r-2}) \subseteq \F$. Therefore $A\cup\{v_{r-2}\}$ is a set of vertices of size $r$ incident with at most $k-1$ hyperedges, therefore Lemma \ref{main} holds. We may assume $d= 1$, and then each $u_i$ is a degree one vertex. 
We may assume that the length of $P$ is at least $r$, otherwise $N_h(v_{r-1})\subset \F$, hence $A\cup \{v_{r-1}\}$ is a vertex set of size $r$ incident with at most $k-1$ hyperedges, therefore Lemma \ref{main} holds. 
    \begin{claim}
    \label{U}
    If there exists a hyperedge $e$, $e \not = e_{r}$ and $e \cap (U\backslash\{v_{r-1}\})\not=\emptyset$ then the vertices in  $e \backslash U$ are only incident with $e$.
    \end{claim}
    \begin{proof}
Suppose without loss of generality $v_1 \in e$, otherwise we can rearrange path. If $e$ is a hyperedge of semi-path $P$, then $e = e_j$  for some $j<r$, otherwise we have a  Berge cycle length at least $k$, a contradiction.
If $e=e_j$, $j<r$ then we already deduced that Claim \ref{U} holds. If $e$ is not a defining hyperedge of semi-path $P$, then consider $P'$ obtain by replacing $e_1$ in $P$ with $e$, from Lemma \ref{first hyperedge},  a vertex in $v \in e\backslash U$  can only be incident with $e,e_2,e_3,\dots,e_{r-1}$, but if $v$ is incident with one of this hyperedges from $e_2,e_3,\dots,e_{r-1}$  then  $e,e_1,e_2,\dots,e_{r-1}$ together with the vertices $v,v_1,v_2,\dots,v_{r-1}$ in some order would be a Berge cycle of length $r$, a contradiction. Finally we have $N_h(e\backslash U)= \{e\}$. \end{proof}
Let $h_1,h_2, \dots,h_p$ be the hyperedges incident with $U\backslash{v_{r-1}}$.
If $\abs{h_i\backslash U} \geq 2$, for some $i$, then $(e_1\cup e_2 \cup \cdots \cup e_{r-2}\cup h_i) \backslash U$ is a set of size at least $r$ incident with $k-1$ hyperedges, hence Lemma \ref{main} holds. Otherwise we have $\abs{h_i\backslash U} = 1$, so this hyperedges form a $r$-star $\S$ with $p\geq r-1$ hyperedges. Every  hyperedge  from $E(\h)\backslash \{e_r, h_1, h_2,\dots,h_p\}$ can only intersect $V(\S)$ in $v_{r-1}$, by setting $\K$ the $r$-graph induce from $\h$ by the vertices  $\{v_{r-1}\}\cup (V(H)\backslash V(S))$ we get a desired partition, therefore Lemma \ref{main} holds. 

\end{itemize}


\section*{Acknowledgment}

The research of first and third authors was partially supported by the National Research, Development and Innovation Office NKFIH, grants  K116769, K117879 and K126853. The research of the third author is partially supported by  Shota Rustaveli National Science Foundation of Georgia SRNSFG, grant number  FR-18-2499.

\end{document}